\documentclass[reqno]{amsart}
\usepackage{amssymb}
\usepackage[dvips]{epsfig}
\usepackage{graphicx}
\usepackage{color}
\usepackage{amsmath}
\usepackage{amssymb}         
\usepackage{amsfonts}
\usepackage{amsthm}
\usepackage{bbm}

\newcommand{\R}{\mathbb R}
\newcommand{\p}{\partial}

\newcommand{\sgn}{{\rm sgn}}

\newcommand{\hil}{\mathcal{H}}

\renewcommand{\a}{\mathfrak{a}}

\newtheorem{theorem}{Theorem}[section]
\newtheorem{lemma}[theorem]{Lemma}
\newtheorem{proposition}[theorem]{Proposition}

\theoremstyle{remark}
\newtheorem{remark}{Remark}[section]
\theoremstyle{definition}

\numberwithin{equation}{section}

\begin{document}

\title[Asymptotic behavior]{Asymptotic behavior of solutions of the dispersive generalized Benjamin-Ono equation}
\author{F. Linares}
\address[F. Linares] {IMPA\\ Estrada Dona Castorina 110, Rio de Janeiro 22460-320, RJ Brazil}
\email{linares@impa.br}
\author{A. Mendez}
\address[A. Mendez] {IMPA\\ Estrada Dona Castorina 110, Rio de Janeiro 22460-320, RJ Brazil}
\email{amendez@impa.br }
\author{G. Ponce}
\address[G. Ponce]{Department  of Mathematics\\
University of California\\
Santa Barbara, CA 93106\\
USA}
\email{ponce@math.ucsb.edu}


\begin{abstract}
We show that for any uniformly bounded in time $H^1\cap L^1$ solution of the dispersive generalized Benjamin-Ono equation,
the limit infimum, as time $t$ goes to infinity,  converges to zero locally in an increasing-in-time region of space of order $t/\log t$.
This result is in accordance with the one established by Mu\~noz and Ponce \cite{MP1} for solutions of the Benjamin-Ono equation.
Similar to solutions of the Benjamin-Ono equation, for a solution of the dispersive generalized Benjamin-Ono equation, with a mild $L^1$-norm 
growth in time, its limit infimum must converge to zero, as time goes to infinity, locally in an increasing on time region of space 
of order depending on the rate of growth of its $L^1$-norm. As a consequence, the existence of  breathers or any other solution for the dispersive
 generalized Benjamin-Ono equation moving with a speed \lq\lq slower" than a soliton is discarded.
In our analysis  the use of commutators expansions is essential.

\end{abstract}

\maketitle




\section{Introduction}

This work is concerned with solutions of the initial value problem (IVP) for the dispersion generalized Benjamin-Ono (DGBO) equation,
\begin{equation}\label{dgbo}
\begin{cases}
\partial_{t}u-D_{x}^{\alpha+1}\partial_{x}u+u\partial_{x}u=0,\hskip15pt x,t\in\mathbb{R}, \;0<\alpha<1, \\
u(x,0)=u_{0}(x),  \\
\end{cases} 
\end{equation}
where $D^s_x$ denotes  the homogeneous derivative of order $s\in\R$:
$$
D^s_x=(-\Delta)^{s/2}\hskip15pt\text{with}\hskip10pt D^s_x=(\mathcal H\,\partial_x)^s,\;\;\;\text{if}\;\;\;n=1
\;\;\text{and}\;\; D^s_x f=c_s\big(|\xi|^s\widehat{f}\,\big)^{\vee},
$$
where $\hil$ denotes the Hilbert transform,
\begin{equation*}
\begin{split}
\hil f(x)&=\frac{1}{\pi} {\rm v.p.}\big(\frac{1}{x}\ast f\big)(x)\\
&=\frac{1}{\pi}\lim_{\epsilon\downarrow 0}\int\limits_{|y|\le \epsilon} \frac{f(x-y)}{y}\,dy=-i\,\sgn(\xi) \widehat{f}(\xi))^{\vee}(x).
\end{split}
\end{equation*}
These equations model vorticity waves in the coastal zone, see \cite{mst} and references therein.

For $\alpha=0$ and $\alpha=1$ the equations in \eqref{dgbo} correspond to the well known Benjamin-Ono (BO) and Korteweg-de Vries (KdV) equations,
respectively.

Even though the equations in \eqref{dgbo} are not completely integrable their solutions satisfy the following conserved quantities,
\begin{equation}\label{laws}
\begin{split}
& I(u)(t)=\int\limits_{\R} u (x,t)\,dx, \hskip20pt M(u)(t)=\int\limits_{\R} u^2(x,t)\,dx,\\
& E(u)(t)=\frac12\int\limits_{\R} |D^{\frac{1+\alpha}{2}}_xu(x,t)|^2\,dx -\frac16\int\limits_{\R} u^3(x,t)\,dx.
\end{split}
\end{equation}

Regarding local well-posedness theory for the IVP \eqref{dgbo}, there is an extensive literature addressing this issue in Sobolev
spaces $H^s(\R)= (1-\p_x^2)^sL^2(\R)$. 
See for instance \cite{CoKeSt, guo, Hr, HIKK, KePoVe, mr} and references therein. We shall recall that  proving the well-posedness for the IVP 
\eqref{dgbo} by direct contraction principle, the principal obstruction is the loss of derivative from the nonlinearity. It was proved by Molinet, 
Saut and Tzvetkov \cite{mst} that if $0\le \alpha <1$ then $H^s$ assumption alone on the initial data is insufficient for a proof
of local well-posedness of \eqref{dgbo} via Picard iteration by showing the solution mapping fails to be
$C^2$ smooth from $H^s$ to $C([0,T];H^s)$ at the origin for any $s$.  The methods of proof of the results above are based in compactness techniques. Concerning global well-posedness (GWP) for the IVP \eqref{dgbo} for initial data in $H^s(\R)$,  Molinet and Ribaud \cite{mr} established  a global result for $s>\frac{1+\alpha}{2}$, $\alpha\in(0,1)$, for initial data satisfying a constrain on the lower frequencies. Herr in \cite{Hr} proved GPW in $H^s(\R)$, $s\ge 0$ requiring that the initial data have an extra property in low frequencies to apply the contraction principle successfully. In \cite{guo}  using the argument introduced in \cite{IKT}  Guo showed GWP for $s>\frac{1+\alpha}{2}$, $\alpha\in (1/3,1)$, without any restriction on the initial data. Finally, Herr, Ionescu, Kenig and Koch, in \cite{HIKK}, establish GWP for initial data in $H^s(\R)$, $s\ge 0$ by using a paradifferential gauge.


\vspace{3mm}

Traveling wave solutions of \eqref{dgbo}  are solutions of the form
$$
u(x,t)=c^{1+\alpha}\,\varphi_{\alpha}(c(x-c^{1+\alpha}t)),\hskip15pt c>0,
$$
$\varphi_{\alpha}$ is even, positive and decreasing fo $x>0$. In the case of the KdV, $\alpha=1$, one has that
$$
\varphi_{1}(x) =\frac{3}{2}\,{\rm sech}^2\Big(\frac{x}{2}\Big).
$$
In the case of the BO equation, $\alpha=0$, one has that
$$
\varphi_0(x)=\frac{4}{1+x^2}.
$$
For $\,\alpha\in(0,1)$  the existence of the ground state was established in \cite{MIW} by variational arguments.
More recently, their uniqueness was established in \cite{Lezman}, although no explicit formula is known. In \cite{KeMaRo}   the following upper bound for the decay of the ground state was deduced
$$
\varphi(x)\leq \frac{c_{\alpha}}{(1+x^2)^{1+\alpha/2}},\hskip15pt \alpha\in(0,1).
$$

\vspace{3mm}

Our aim in this work is the study of the asymptotic behavior  of solutions of the IVP \eqref{dgbo}. Recently, Mu\~noz and Ponce \cite{MP1}
examined this issue for solutions of the IVP associated to the BO equation,
\begin{equation}\label{bo}
\partial_{t}u+\hil\partial^2_{x}u+u\partial_{x}u=0, \hskip15pt x,t\in\mathbb{R}.
\end{equation}
Their main goal was to establish the \lq\lq location'' of the $H^{1/2}$-norm of solutions which are globally bounded as times evolves. For this
purpose they assumed the following decay:

There exist $a\in [0,\frac{1}{2})$ and $c_0>0$ such that for any $T>0$
\begin{equation}\label{decay-bo}
\underset{t\in[0,T]}{\sup} \int_{\R} |u(x,t)|\,dx \le  c_0 \langle T\rangle^a, \hskip10pt \langle T\rangle=(1+T^2)^{1/2}.
\end{equation}

It was established in \cite{MP1} :

\begin{theorem}\label{main-bo}
Let $u=u(x,t)$ be a solution of the IVP associated to \eqref{bo} such that
\begin{equation}\label{main1-bo}
u\in C(\R: H^1(\R))\cap L^{\infty}_{\rm loc} (\R:L^1(\R))
\end{equation}
satisfying \eqref{decay}. Then

\begin{equation}\label{main2-bo}
\begin{split}
\int\limits_{\left\{t\gg 1\right\}}\!\!\!\frac{1}{t\log t}\Big\{\int\limits_{\R} \Big(u^2+\big(D_{x}^{\frac{1}{2}}u\big)^{2}\Big)\,\phi_{\alpha}'\Big(\frac{x}{\lambda(t)}\Big) (x,t)\,dx\Big\}\,dt< \infty.
\end{split}
\end{equation}

Hence
\begin{equation}\label{main3-bo}
\underset{t\uparrow \infty}{\liminf}\int_{\mathbb{R}} \Big(u^2+\big (D_x^{\frac{1}{2}}u\big)^{2}\Big)(x,t)\,\phi_{\alpha}'\Big(\frac{x}{\lambda(t_{n})}\Big)\,dx= 0
\end{equation}
with
\begin{equation}\label{main4-bo}
\lambda(t)=\frac{ct^b}{\log t},\hskip10pt a+b=1, \text{\hskip10pt and}\hskip10pt \phi'(x)=\frac{1}{1+ x^2}
\end{equation}
for any fixed $c>0$.
\end{theorem}

\vspace{3mm}

In this case we also shall assume decay as in \eqref{decay-bo}, that is: There exist $a\in [0,\frac{1}{2+\alpha})$ and $c_0>0$ such that for any $T>0$
\begin{equation}\label{decay}
\underset{t\in[0,T]}{\sup} \int_{\R} |u(x,t)|\,dx \le  c_0 \langle T\rangle^a, \hskip10pt \langle T\rangle=(1+T^2)^{1/2}.
\end{equation}
Our main result in this work is as follows.
\vspace{3mm}

\begin{theorem}\label{main}
Let $u=u(x,t)$ be a solution of the IVP \eqref{dgbo} such that
\begin{equation}\label{main1}
u\in C(\R: H^1(\R))\cap L^{\infty}_{\rm loc} (\R:L^1(\R))
\end{equation}
satisfying \eqref{decay}. Then

\begin{equation}\label{main2}
\begin{split}
\int\limits_{\left\{t\gg 1\right\}}\!\!\!\frac{1}{t\log t}\Big\{\int\limits_{\R} \Big(u^{2}\!+\!\big(D_{x}^{\frac{\alpha+1}{2}}u\big)^{2}\!+\!\big(\hil D_{x}^{\frac{\alpha+1}{2}}u\big)^{2}\Big)(x,t)\,\phi_{\alpha}'\Big(\frac{x}{\lambda(t)}\Big)dx\Big\}\,dt< \infty.
\end{split}
\end{equation}

Therefore
\begin{equation}\label{main3}
\underset{t\uparrow \infty}{\liminf}\int_{\mathbb{R}} \Big(u^{2}+\big(D_{x}^{\frac{\alpha+1}{2}}u\big)^{2}+\big(\hil D_{x}^{\frac{\alpha+1}{2}}u\big)^{2}\Big)(x,t)\,\phi_{\alpha}'\Big(\frac{x}{\lambda(t_{n})}\Big)\,dx= 0
\end{equation}
with
\begin{equation}\label{main4}
\lambda(t)=\frac{ct^b}{\log t},\hskip10pt a+b=1, \text{\hskip10pt and}\hskip10pt \phi_{\alpha}'(x)=\frac{1}{\langle x\rangle^{\alpha+2}},
\end{equation}
for any fixed $c>0$.
\end{theorem}

\begin{remark} As in \cite{MP2} our approach was inspired by the works of Kowalczyk, Martel and Mu\~noz  \cite{KMM}-\cite{KMM1} 
concerning the decay of solutions in  $1+1$ dimensional scalar field models. 
\end{remark}

\begin{remark} Theorem \ref{main} proves that the limit infimum, as time tends to infinity, for any uniformly bounded (or with mild
growth) $L^1$ solution  of the DGBO equation converges to zero locally in an increasing in time region of the space. This eliminates the existence of solutions moving with a speed slower than a traveling wave. In this regard, from the argument in \cite{MP1}, i.e. multiplying an appropriate solution $u(x,t)$ of the equation in \eqref{dgbo} by $x$ and integrating the result, one gets
\begin{equation}
\label{brea}
\frac{d\;}{dt}\int xu(x,t)dx-\frac{1}{2}\,\int u^2(x,t)dx=0.
\end{equation}
This tells us that any appropriate  non-trivial solution $u(x,t)$ cannot be time periodic (breather).

 To justify the computations in \eqref{brea} it suffices to have an initial data $u_0$ satisfying : $\,u_0,\,D^{1+\alpha}_x\partial_xu_0\in L^2(\mathbb R:x^4dx)\cap H^7(\mathbb R)$, (see \cite{flp}).

\end{remark}

\begin{remark} We recall that in \cite{MP2}  for the case of the KdV a similar result was established in a space region with lesser growth but with the whole limit as $\,t \uparrow \infty$ instead of the limit infimum.
\end{remark}

\begin{remark} The argument of proof given here generalize those  in \cite{MP1} for the case $\alpha=0$.
\end{remark}

\begin{remark} If we consider solutions of the IVP associated to the fractional KdV equation,
\begin{equation}\label{fkdv}
\left\{
\begin{array}{ll}
\partial_{t}u+D_{x}^{\alpha}\partial_{x}u+u\partial_{x}u=0, & x,t\in\mathbb{R}, \;0<\alpha<1, \\
u(x,0)=u_{0}(x).&  \\
\end{array} 
\right.
\end{equation}
A result like in Theorem \ref{main} would be true whenever the property described in Lemma \ref{lem2} below holds for
$\alpha>6/7$ (see \cite{mpv}) where global solutions are known.
\end{remark}

This paper is organized as follows:  In Section 2 we describe the main technical tools we will use to establish our main result. The
proof of Theorem \ref{main} will be given in Section 3.

\medspace

\section{Preliminaries}

\subsection{{Commutator Expansions}}
We start by  presenting  several auxiliary results obtained by Ginibre and Velo \cite{{GV1}}, \cite{GV2} useful in our analysis.

Let $\a=2\mu+1>1$,  let  $n$ be a nonnegative integer and $f$ be  a smooth  function with suitable  decay at infinity, for instance,  
$f'\in C^{\infty}_{0}(\mathbb{R})$. 

We define the operator 

\begin{equation}\label{e2}
R_{n}(\a)=-\left[\mathcal{H} D^{\a}; f\right]-\frac{1}{2}\left(P_{n}(\a)-\mathcal{H}P_{n}(\a)\mathcal{H}\right),
\end{equation}
where 
\begin{equation}\label{e3}
P_{n}(\a)=\a\sum_{0\leq j\leq n}c_{2j+1}(-1)^{j}4^{-j}D^{\mu-j}f^{(2j+1)}D^{\mu-j},
\end{equation}
and the constants $c_{2j+1}$ are given by the  following  formula
\begin{equation}
c_{1}=1\quad\mbox{and}\quad c_{2j+1}=\frac{1}{(2j+1)!}\prod_{0\leq k<j}\left(\a^{2}-\left(2k+1\right)^{2}\right).
\end{equation}

\begin{proposition}\label{p1}
	Let $n$  be a non-negative  integer,   $\a\geq  1,\,$ and   $ \sigma\geq 0,$  be such that 
	\begin{equation}\label{eq21}
	2n+1\leq \a+2\sigma\leq2n+3.
	\end{equation}
	
	Then 
	\begin{itemize}
		\item[(a)] The operator $D^{\sigma}R_{n}(\a)D^{\sigma}$ is bounded in $L^{2}$ with norm 
		\begin{equation}\label{eq98}
		\left\|D^{\sigma}R_{n}(\a)D^{\sigma}h\right\|_{2}\leq C(2\pi)^{-1/2}\left\|\widehat{\left(D^{\a+2\sigma}f\right)}\right\|_{1}\|h\|_{2}.
		\end{equation}	
		If $a\geq 2n+1,$ one can take  $C=1.$
		
		\item[(b)] Assume in addition  that
		\begin{equation*}
		2n+1\leq \a+2\sigma<2n+3.
		\end{equation*}
		Then the operator ${\displaystyle D^{\sigma}R_{n}(\a)D^{\sigma}}$ is compact in $L^{2}(\mathbb{R}).$
	\end{itemize}
\end{proposition}

\begin{proof}
	See  Proposition 2.2 in \cite{GV2}.
\end{proof}

\subsection{Technical Tools}

We will first consider the following functions which are the key ingredient in the energy estimates.

For $\alpha\in(0,1)$ fixed,  we define the function
\begin{equation}
\phi_{\alpha}(x):=\int_{-\infty}^{x}\frac{\mathrm{d}s}{\langle s\rangle^{\alpha+2}}, \quad x\in \mathbb{R}
\end{equation}
and the bracket above denotes  $\langle s\rangle:=\sqrt{1+s^{2}}.$

Notice that for $\alpha\in (0,1)$ the function $\phi_{\alpha}$ is uniformly bounded. More  precisely, it satisfies
\begin{equation}\label{e6}
\phi_{\alpha}(x)\leq 2\left(\frac{1}{\alpha+3}\right)\quad \mbox{for all}\quad x\in \mathbb{R}.
\end{equation}

\begin{lemma}\label{lem1}
For any $\alpha\in (0,1)$ the Fourier transform of the function $\phi_{\alpha}'$  is given by
	
\begin{equation}
\widehat{\phi_{\alpha}'}(\xi)=\frac{\sqrt{\pi}}{\Gamma\left(\frac{\alpha+2}{2}\right)}\int_{0}^{\infty} e^{-s}s^{\frac{\alpha-1}{2}} e^{-\frac{\pi^{2}\xi^{2}}{s}}\,\mathrm{d}s.
\end{equation}
\end{lemma}

\begin{proof}
	Observe that $\phi_{\alpha}'$ can be rewritten as
	\begin{equation}
	\phi_{\alpha}'(x)=\frac{1}{\Gamma\left(\frac{\alpha+2}{2}\right)}\int_{0}^{\infty} e^{-\left(1+x^{2}\right)s}s^{\frac{\alpha}{2}}\,\mathrm{d}s,
	\end{equation}
with this at hand and  Fubini's theorem it follows that 

\begin{equation}
\begin{split}
\widehat{\phi_{\alpha}'}(\xi)&=\int_{\mathbb{R}}e^{-2i\pi x\xi}\,\phi_{\alpha}'(x)\,\mathrm{d}x\\
	&=\frac{1}{\Gamma\left(\frac{\alpha+2}{2}\right)}\int_{\mathbb{R}}\int_{0}^{\infty}e^{-2i\pi x\xi}e^{-\left(1+x^{2}\right)s}s^{\frac{\alpha}{2}}\,\mathrm{d}s\mathrm{d}x\\
	&=\frac{1}{\Gamma\left(\frac{\alpha+2}{2}\right)}\int_{0}^{\infty}e^{-s}s^{\frac{\alpha}{2}}\int_{\mathbb{R}}e^{-2\pi ix\xi}e^{-x^{2}s}\,\mathrm{d}x\,\mathrm{d}s\\
	&=\frac{1}{\Gamma\left(\frac{\alpha+2}{2}\right)}\int_{0}^{\infty}e^{-s}s^{\frac{\alpha}{2}}\left(\sqrt{\frac{\pi}{s}}e^{-\frac{\pi^{2} \xi^{2}}{s}}\right)\,\mathrm{d}s\\
	&=\frac{\sqrt{\pi}}{\Gamma\left(\frac{\alpha+2}{2}\right)}\int_{0}^{\infty}e^{-s}s^{\frac{\alpha-1}{2}}e^{-\frac{\pi^{2} \xi^{2}}{s}}\,\mathrm{d}s.
\end{split}
\end{equation}
Therefore,
\begin{equation}
\widehat{\phi_{\alpha}'}(\xi)=\frac{\sqrt{\pi}}{\Gamma\left(\frac{\alpha+2}{2}\right)}\int_{0}^{\infty}e^{-s}s^{\frac{\alpha-1}{2}}e^{-\frac{\pi^{2} \xi^{2}}{s}}\,\mathrm{d}s, \quad\mbox{for}\quad \alpha\in(0,1).
\end{equation}
\end{proof}

The next lemma contains a useful interpolation estimate and a fractional Leibniz' rule needed in our arguments.
\begin{lemma}\label{lem3}
	Let $\alpha\in [0,1]$, it holds that
	\begin{equation}\label{e1.1}
	\|f\|_{p}\le c_p\, \|f\|_{2}^{1-\frac{p-2}{(\alpha+1)\,p}}\|D^{\frac{1+\alpha}{2}}_{x} f\|_{2}^{\frac{p-2}{(\alpha+1)\,p}}, \hskip10pt 2\le p<\infty,
	\end{equation}
	
	\begin{equation}\label{e2.1}
	\left\|D^{\frac{\alpha+1}{2}}_{x}(fg)-g D^{\frac{\alpha+1}{2}}_xf\right\|_{2} \lesssim \|f\|_{4} \|D^{\frac{\alpha+1}{2}}_xg\|_{4}.
	\end{equation}
\end{lemma}

\begin{proof}
	
	The Gagliardo-Nirenberg inequality \eqref{e1.1} follows from complex interpolation and Sobolev embedding.
	
	The estimate \eqref{e2.1} is derived from the Leibniz rule for fractional derivatives in \cite{kpv-93} (Theorem A.8).
\end{proof}

\vspace{3mm} 

Employing the argument to establish  (2.30) in Lemma 3  of  \cite{km}  we obtain.

\begin{lemma}\label{lem2} Let $u$ be a solution of the IVP \eqref{dgbo} and $\phi_{\alpha}$  the function defined in \eqref{dgbo}.	It holds that
	\begin{equation}\label{e3.1}
	\int_{\mathbb{R}} | v(x-\rho(t), t)|^3 \phi_{\alpha}'(\tilde{x})\,dx \le c\left(\|v_0\|_{\frac{\alpha+1}{2},2}\right)\int_{\mathbb{R}} v^2(x-\rho(t), t) \phi_{\alpha}'(\tilde{x})\, dx.
	\end{equation}
where $\rho(t)\in C^1(\R)$.
\end{lemma}

\begin{proof} Let $\chi:\mathbb{R}\longrightarrow\mathbb{R}$ be a $C^{\infty}$ function such that $\chi\equiv 1$ on $[0,1]$, $\chi\equiv 0$ in $(-\infty,-1]\cup[2,\infty)$, and $\chi\le 1$ on
	$\mathbb{R}$.
	
	Define $\chi_n(x)=\chi(x-n)$.
	
	Employing the Gagliardo-Nirenberg inequality \eqref{e1.1} we have
	\begin{equation}\label{e4.1}
	\begin{split}
	\int &| v(x)|^3 \phi_{\alpha}'(x-y_0)\,dx  \\
	&\le \sum_{n\in\mathbb{Z}} \int_n^{n+1} |v|^3\phi_{\alpha}'(x-y_0)\,dx\\
	&\le \sum_{n\in\mathbb{Z}} \Big(\int_{\mathbb{R}} |v|^3\chi_n^3\Big)\;\Big( \underset{[n-y_0,n+1-y_0]}{\sup} \phi_{\alpha}'\Big)\\
	&\le \sum_{n\in\mathbb{Z}} \Big(\int_{\mathbb{R}} \left|D^{\frac{\alpha+1}{2}}_{x}\left(v\chi_n\right)\right|^2\Big)^{\frac{1}{2(\alpha+1)}} \Big(\int_{\mathbb{R}} \left(v\chi_n\right)^2\Big)^{\frac{3\alpha+2}{2(\alpha+1)}} \Big(
	\underset{[n-y_0,n+1-y_0]}{\sup} \phi_{\alpha}'\Big).
	\end{split}
	\end{equation}
	On the other hand, the estimate \eqref{e2}, Sobolev embedding and Sobolev spaces properties yield
	\begin{equation}\label{e5.1}
	\begin{split}
	\left\| D^{\frac{\alpha+1}{2}}_x (v\chi_n)\right\|_{2}^{\frac{1}{\alpha+1}} &\le \left\|D^{\frac{\alpha+1}{2}}_x(v\chi_n)-vD^{\frac{\alpha+1}{2}}_x \chi_n\right\|_{2}^{\frac{1}{\alpha+1}}
	+\left\|vD^{\frac{\alpha+1}{2}}_x \chi_n\right\|_{2}^{\frac{1}{\alpha+1}}\\
	&\le  c\|v\|_{4}^{\frac{1}{\alpha+1}}\left\|D^{\frac{\alpha+1}{2}}_x \chi_n\right\|_{4}^{\frac{1}{\alpha+1}}\\
	&\lesssim \|v\|_{\frac{1}{2},2}^{\frac{1}{\alpha+1}}\\
	&\lesssim \|v\|_{\frac{\alpha+1}{2},2}^{\frac{1}{\alpha+1}}.
	\end{split}
	\end{equation}
	By the local well-posedness theory $\|v\|_{L^{\infty}_T H^{\frac{\alpha+1}{2}}}\le c\,\|v_0\|_{H^{\frac{\alpha+1}{2}}}$. Thus coming back to estimate \eqref{e4.1} we
	have
	\begin{equation}\label{e6.1}
	\begin{split}
	\int_{\mathbb{R}} |v|^3 &\phi_{\alpha}'(x-y_0)\,\mathrm{d}x \le c(\|v_0\|_{\frac{\alpha+1}{2},2}) \sum_{n\in\mathbb{Z}}\Big(\int_{\mathbb{R}} \left(v\chi_n\right)^2\Big)^{\frac{3\alpha+2}{2(\alpha+1)}} 
	\left(\underset{[n-y_0,n+1-y_0]}{\sup} \phi_{\alpha}'\right)\\
	&\le c(\|v_0\|_{\frac{\alpha+1}{2},2}) \sum_{n\in \mathbb{Z}}\Big(\int_{\mathbb{R}} (v\chi_n)^2\Big) \left(\underset{[n-y_0,n+1-y_0]}{\sup} \phi_{\alpha}'\right)\\
	&\le  c(\|v_0\|_{\frac{\alpha+1}{2},2}) \int_{\mathbb{R}} v^2 \phi_{\alpha}'(x-y_0).
	\end{split}
	\end{equation}
	where we use that $\frac{3\alpha+2}{\alpha+1}=2+\frac{\alpha}{\alpha+1}>2,$ for any $\alpha>0$ and 
	\begin{equation}\label{e7.1}
	\underset{[y,y+4]}{\sup} \phi'_{\alpha} \le c\underset{[y, y+4]}{\inf} \phi'_{\alpha}.
	\end{equation}
	
	This shows \eqref{e3.1}.

\end{proof}

\section{Proof of Theorem \ref{main} }
\begin{flushleft}
	\fbox{{\sc Assumption:}}
\end{flushleft}
We will assume that  there exist $a>0$ and $c_{0}>0$ such that  for all $T>0$
\begin{equation}\label{a1}
\sup_{t\in[0,T]}\int_{\mathbb{R}}|u(x,t)|\,\mathrm{d}x\leq c_{0}\langle T\rangle^{a},
\end{equation}
throughout our analysis we will impose some restrictions on $a.$

We also define the functions
\begin{equation*}
\lambda(t)=\frac{ct^{b}}{\log t}\quad \mbox{for any $c>0,$  },
\end{equation*}
 with $a+b=1.$
\begin{flushleft}
	\fbox{{\sc Step 1:}}
\end{flushleft}

We first multiply the equation in \eqref{dgbo} by 
\begin{equation*}
\frac{1}{t^{a}\log^{2}t}\phi_{\alpha}\left(\frac{x}{\lambda(t)}\right)
\end{equation*}
to obtain after integration the following identity:
\begin{equation}\label{e5}
\begin{split}
&\frac{\mathrm{d}}{\mathrm{d}t}\int_{\mathbb{R}}\frac{1}{t^{a}\log^{2} t}\,\phi_{\alpha}\left(\frac{x}{\lambda(t)}\right)u(x,t)\,\mathrm{d}x\\
&\underbrace{-\int_{\mathbb{R}}\left(\frac{1}{t^{a}\log^{2}t}\right)'\phi_{\alpha}\left(\frac{x}{\lambda(t)}\right)u(x,t)\,\mathrm{d}x}_{A_{1}(t)}\\
&\underbrace{+\frac{\lambda'(t)}{\lambda(t)t^{a}\log^{2}t}\int_{\mathbb{R}}\left(\frac{x}{\lambda(t)}\right)\phi_{\alpha}'\left(\frac{x}{\lambda(t)}\right)u(x,t)\mathrm{d}x}_{A_{2}(t)}\\
&\underbrace{-\frac{1}{t^{a}\log^{2}t}\int_{\mathbb{R}}\phi_{\alpha}\left(\frac{x}{\lambda(t)}\right)D_{x}^{\alpha+1}\partial_{x}u(x,t)\mathrm{d}x}_{A_{3}(t)}\\
&\underbrace{-\frac{1}{2\lambda(t)t^{a}\log^{2}t}\int_{\mathbb{R}}\phi_{\alpha}'\left(\frac{x}{\lambda(t)}\right)u^{2}(x,t)\,\mathrm{d}x}_{A_{4}(t)}=0.
\end{split}
\end{equation}
First we handle $A_{1}.$  Since 
\begin{equation}\label{e7}
\left(\frac{1}{t^{a}\log^{2}t}\right)'=-\left(\frac{a\log t +2}{t^{a+1}\log^{3}t}\right)\sim \frac{1}{t^{a+1}\log^{2}t}\quad\mbox{for} \quad t\gg 1,
\end{equation}
then  the term $A_{1}$ can be controlled by using  the assumption \eqref{a1} as follows
\begin{equation*}
\begin{split}
|A_{1}(t)|&\leq \left(\frac{2}{\alpha+3}\right)\left(\frac{1}{t^{a+1}\log^{2}t}\right)\int_{\mathbb{R}} |u(x,t)|\,\mathrm{d}x\\
&\lesssim \frac{1}{t^{}\log^{2}t}\in L^{1}\left(\left\{t\gg 1\right\}\right). 
\end{split}
\end{equation*}

To estimate  $A_{2}$ we notice that 
$x\phi_{\alpha}'(x)$ is a bounded function. So that,
\begin{equation*}
\frac{x}{\lambda(t)}\phi_{\alpha}'\left(\frac{x}{\lambda(t)}\right)\in L^{\infty}(\mathbb{R})\quad \mbox{uniformly for } t\gg1.
\end{equation*}
Additionally, the function $\lambda(t)$ satisfies 
\begin{equation*}
\frac{\lambda'(t)}{\lambda(t)}=\frac{\log t -1}{t\log t}\sim \frac{1}{t}\quad\mbox{for}\quad t\gg 1.
\end{equation*}
Therefore, 
\begin{equation*}
|A_{2}(t)|\lesssim\frac{1}{t^{a+1}\log^{2}t}\int_{\mathbb{R}}|u(x,t)|\,\mathrm{d}x\lesssim  \frac{1}{t\log^{2}t}\in L^{1}\left(\left\{t\gg 1\right\}\right).
\end{equation*}

In regards $A_{3}$ we have after apply integration by parts, Plancherel's identity and H\"{o}lder's inequality 
\begin{equation}\label{e4}
\begin{split}
A_{3}(t)&=\frac{1}{\lambda(t)t^{a}\log^{2}t}\int_{\mathbb{R}}\phi_{\alpha}'\Big(\frac{x}{\lambda(t)}\Big)D_{x}^{\alpha+1}u(x,t)\,\mathrm{d}x\\
&=\frac{1}{\lambda(t)t^{a}\log^{2}t}\int_{\mathbb{R}}D_{x}^{\alpha+1}\Big(\phi_{\alpha}'\Big(\frac{x}{\lambda(t)}\Big)\Big)u(x,t)\,\mathrm{d}x\\
&\leq\frac{1}{\lambda(t)t^{a}\log^{2}t} \left\|D_{x}^{\alpha+1}\Big(\phi_{\alpha}'\Big(\frac{x}{\lambda(t)}\Big)\Big)\right\|_{\infty}\|u(t)\|_{1}\\
&\leq\frac{1}{\lambda(t)t^{a}\log^{2}t}\left\|\widehat{D_{x}^{\alpha+1}\Big(\phi_{\alpha}'\Big(\frac{x}{\lambda(t)}\Big)\Big)}(\xi)\right\|_{1}\|u(t)\|_{1}.
\end{split}
\end{equation}

To show that the term  $A_{3}$ in $L^{1}_{t}$ we need to know how is the behavior of the fractional derivative above. 
In this order, we estimate this term  as follows:
first notice that  by Minkowski's integral inequality
 \begin{equation*}
\begin{split}
&\left\|\widehat{D_{x}^{\alpha+1}\Big(\phi_{\alpha}'\Big(\frac{x}{\lambda(t)}\Big)\Big)}(\xi)\right\|_{1}\\
&\quad=\frac{1}{|\lambda(t)|^{\alpha+1}}\int_{\mathbb{R}}|\xi|^{\alpha+1}|\widehat{\phi_{\alpha}'}(\xi)|\,\mathrm{d}\xi\\
&\quad =\frac{\sqrt{\pi}}{\Gamma\left(\frac{\alpha+2}{2}\right)|\lambda(t)|^{\alpha+1}}\int_{\mathbb{R}}\left|\int_{0}^{\infty}e^{-s}s^{\frac{\alpha-1}{2}}|\xi|^{\alpha+1}e^{-\frac{\pi^{2}\xi^{2}}{s}}\,\mathrm{d}s\right|\mathrm{d}\xi\\
&\quad \leq\frac{\sqrt{\pi}}{\Gamma\left(\frac{\alpha+2}{2}\right)|\lambda(t)|^{\alpha+1}}\int_{0}^{\infty}\int_{\mathbb{R}}e^{-s}s^{\frac{\alpha-1}{2}}|\xi|^{\alpha+1}e^{-\frac{\pi^{2}\xi^{2}}{s}}\,\mathrm{d}\xi\mathrm{d}s\\
&\quad =\frac{2\sqrt{\pi}}{\Gamma\left(\frac{\alpha+2}{2}\right)|\lambda(t)|^{\alpha+1}}\int_{0}^{\infty}e^{-s}s^{\frac{\alpha-1}{2}}\left(\int_{0}^{\infty}\xi^{\alpha+1}e^{-\frac{\pi^{2}\xi^{2}}{s}}\,\mathrm{d}\xi\right)\mathrm{d}s\\
&\quad =\frac{\Gamma\left(\frac{2\alpha+3}{2}\right)}{\pi^{\frac{2\alpha+3}{2}}|\lambda(t)|^{\alpha+1}}.
\end{split}
\end{equation*}

After inserting the estimates above in   \eqref{e4} we obtain  for $t\gg 1,$ that 
\begin{equation*}
\begin{split}
|A_{3}(t)|&\leq \frac{\Gamma\left(\frac{2\alpha+3}{2}\right)}{\pi^{\frac{2\alpha+3}{2}}\lambda(t)^{\alpha+2}t^{a}\log^{2}t}\|u(t)\|_{1}\\
&\lesssim_{a}\frac{\log^{\alpha} t}{t^{b(\alpha+2)}}\\
&\lesssim_{s} \frac{1}{t^{b(\alpha+2)-\alpha}} \hskip15pt \text{for} \hskip10pt  t\gg1.
\end{split}
\end{equation*}

From the last inequality we need to impose  the condition 
\begin{equation}\label{e8}
b(\alpha+2)-\alpha>1 \hskip10pt \text{or} \hskip10pt b>\frac{\alpha+1}{\alpha+2}.
\end{equation}
Since by hypothesis  $a+b=1$ (as in the case of the Benjamin-Ono equation see Mu\~{n}oz and Ponce \cite{MP1} ) we find that $a,b$ have to  satisfy the inequalities
\begin{equation}\label{e9}
a<\frac{1}{\alpha+2}\qquad \mbox{and} \qquad b>\frac{\alpha+1}{\alpha+2}.
\end{equation}
Under the conditions above we deduce that 
\begin{equation*}
|A_{3}(t)|\lesssim_{\alpha}\frac{1}{\lambda(t)^{\alpha+2}t^{a}\log^{2}t}\in L^{1}\left(\left\{t\gg 1\right\}\right).
\end{equation*}
Finally,  after integrating in time the identity \eqref{e5} combined with the estimates obtained above  we obtain that 
\begin{equation}\label{e10}
\begin{split}
&\int_{\{t\gg 1\}}\frac{1}{\lambda(t)t^{a}\log t}\left(\int_{\mathbb{R}}\phi_{\alpha}'\left(\frac{x}{\lambda(t)}\right)u^{2}(x,t)\,\mathrm{d}x\right)\,\mathrm{d}t\\
&\quad =\int_{\{t\gg 1\}}\frac{1}{t\log t}\left(\int_{\mathbb{R}}\phi_{\alpha}'\left(\frac{x}{\lambda(t)}\right)u^{2}(x,t)\,\mathrm{d}x\right)\,\mathrm{d}t\\
&\quad< \infty.
\end{split}
\end{equation}
\begin{flushleft}
	\fbox{{\sc Step 2:}}
\end{flushleft}
We multiply the equation in \eqref{dgbo} by 
\begin{equation}
\frac{1}{t^{a}\log^{2}t}\phi_{\alpha}\left(\frac{x}{\lambda(t)}\right)u(x,t)
\end{equation}
to obtain after integration the following identity:
\begin{equation}
\begin{split}
&\frac{\mathrm{d}}{\mathrm{d}t}\int_{\mathbb{R}}\frac{1}{t^{a}\log^{2} t}\,\phi_{\alpha}\left(\frac{x}{\lambda(t)}\right)u^{2}(x,t)\,\mathrm{d}x\\
&\underbrace{-\int_{\mathbb{R}}\left(\frac{1}{t^{a}\log^{2}t}\right)'\phi_{\alpha}\left(\frac{x}{\lambda(t)}\right)u^{2}(x,t)\,\mathrm{d}x}_{A_{1}(t)}\\
&\underbrace{+\frac{1}{t^{a}\log^{2}t}\int_{\mathbb{R}}\left(\frac{x}{\lambda(t)}\right)\left(\frac{\lambda'(t)}{\lambda(t)}\right)\phi_{\alpha}'\left(\frac{x}{\lambda(t)}\right)u^{2}(x,t)\mathrm{d}x}_{A_{2}(t)}\\
&\underbrace{-\frac{1}{t^{a}\log^{2}t}\int_{\mathbb{R}}\phi_{\alpha}\left(\frac{x}{\lambda(t)}\right)u(x,t)D_{x}^{\alpha+1}\partial_{x}u(x,t)\mathrm{d}x}_{A_{3}(t)}\\
&\underbrace{-\frac{1}{3\lambda(t)t^{a}\log^{2}t}\int_{\mathbb{R}}\phi_{\alpha}'\left(\frac{x}{\lambda(t)}\right)u^{3}(x,t)\,\mathrm{d}x}_{A_{4}(t)}=0.
\end{split}
\end{equation}
To estimate $A_{1}$ we use  that $\phi_{\alpha }$ is uniformly bounded (see \eqref{e6}),  that combined with  \eqref{e7} and the fact that the mass is a conserved quantity   yield
\begin{equation}
|A_{1}(t)|\lesssim_{\|u_{0}\|_{2}} \frac{1}{t^{a+1}\log^{2}t}\in L^{1}\left(\left\{t\gg 1 \right\}\right).
 \end{equation}
 In regards $A_{2}$, again  notice that 
 $x\phi_{\alpha}'(x)$ is a bounded function. So that,
 \begin{equation}
 \frac{x}{\lambda(t)}\phi_{\alpha}'\left(\frac{x}{\lambda(t)}\right)\in L^{\infty}(\mathbb{R})\quad \mbox{uniformly for } t\gg1, 
 \end{equation}
 and  the function $\lambda(t)$ satisfies 
 \begin{equation}
 \frac{\lambda'(t)}{\lambda(t)}=\frac{\log t -1}{t\log t}\sim \frac{1}{t}\quad\mbox{for}\quad t\gg 1.
 \end{equation}
 Therefore, 
 \begin{equation}
 \begin{split}
 |A_{2}(t)|&\lesssim_{\alpha}\frac{\lambda'(t)}{\lambda(t) \,t^{a}\log^{2}t}\int_{\mathbb{R}}u^{2}(x,t)\,\mathrm{d}x\\
 &\lesssim_{T,a,\alpha} \frac{1}{t^{a+1}\log^{2}t}\in L^{1}\left(\left\{t\gg 1\right\}\right).
 \end{split}
 \end{equation}
 Concerning $A_{3}$ we obtain after apply integration by parts 
and Plancherel's identity that 
 \begin{equation}
 A_{3}(t)=-\frac{1}{2t^{a}\log^{2}t}\int_{\mathbb{R}} u(x,t)\left[\mathcal{H}D_{x}^{\alpha+2}; \phi_{\alpha}\left(\frac{x}{\lambda(t)}\right)\right]u(x,t)\,\mathrm{d}x
 \end{equation}
which after  use the commutator decomposition  \eqref{e2}, it can be rewritten as 
\begin{equation}
\begin{split}
A_{3}(t)&=\frac{1}{2t^{a}\log^{2}t}\int_{\mathbb{R}}u(x,t)\left(R_{n}(\alpha+2)u\right)(x,t)\,\mathrm{d}x\\
&\quad +\frac{1}{4t^{a}\log^{2}t}\int_{\mathbb{R}}u(x,t)\left(P_{n}(\alpha+2)u\right)(x,t)\,\mathrm{d}x\\
&\quad -\frac{1}{4t^{a}\log^{2}t}\int_{\mathbb{R}}u(x,t)\left(
\mathcal{H}P_{n}(\alpha+2)\mathcal{H}u\right)(x,t)\,\mathrm{d}x\\
&=A_{3,1}(t)+A_{3,2}(t)+A_{3,3}(t).
\end{split}
\end{equation}
First we handle $A_{3,1}.$ We will  fix   $n$  satisfying the inequality
\begin{equation}
2n+1\leq \alpha+2\leq 2n+3;
\end{equation}
  from where we obtain $n=0.$ For this particular value of $n$ the remainder term $R_{0}(\alpha+2)$  maps $L^{2}(\mathbb{R})$ into $L^{2}(\mathbb{R}),$ more precisely
  \begin{equation}
  \left\|R_{0}(\alpha+2)f\right\|_{2}\leq \frac{c}{\sqrt{2\pi}}\|f\|_{2}\left\|\widehat{D_{x}^{\alpha+2}\phi_{\alpha}\left(\frac{x}{\lambda(t)}\right)}
  	(\xi)\right\|_{1}
  \end{equation}
  for $f$ in a suitable class.
  
  Therefore, for $t\gg 1$ we obtain by Lemma \ref{lem1}  that
  \begin{equation}
 \begin{split}
 |A_{3,1}|(t)&\lesssim\left|\frac{1}{2t^{a}\log^{2}t}\right|\|u(t)\|_{2}^{2}\left\|\widehat{D_{x}^{\alpha+2}\phi_{\alpha}\left(\frac{x}{\lambda(t)}\right)}\right\|_{1}\\
 &\lesssim \left|\frac{c_{\alpha+1}}{\lambda^{\alpha+2}(t)t^{a}\log^{2}t}\right|\|u_{0}\|_{2}^{2}\Gamma\left(\frac{2\alpha+3}{2}\right)\\
 &\lesssim_{\alpha}\left|\frac{1}{\lambda^{\alpha+2}(t)t^{a}\log^{2}t}\right|\|u_{0}\|_{2}^{2}\\
 &\lesssim_{\alpha,\|u_{0}\|_{2}} \frac{\log^{\alpha}t}{t^{b(\alpha+2)+a}}.
 \end{split}
 \end{equation}
From where we need to impose the  condition suggested in \eqref{e8}-\eqref{e9}  
to obtain  that $|A_{3,1}(t)|\in L^{1}\left(\left\{t\gg 1 \right\}\right).$
 
 Since we  fixed $n=0,$ we get after  replacing $P_{0}$ into $A_{3,2}$ and $A_{3,3}$ that 
 \begin{equation}
 A_{3,2}(t)=\left(\frac{\alpha+2}{4\lambda(t)t^{a}\log^{2}t}\right)
 \int_{\mathbb{R}}\left(
 D_{x}^{\frac{\alpha+1}{2}}u\right)^{2}(x,t)\phi_{\alpha}'\left(\frac{x}{\lambda(t)}\right)\,\mathrm{d}x
 \end{equation}
and 
\begin{equation}
 A_{3,3}(t)=\left(\frac{\alpha+2}{4\lambda(t)t^{a}\log^{2}t}\right)\int_{\mathbb{R}}\left(
\mathcal{H}D_{x}^{\frac{\alpha+1}{2}}u\right)^{2}(x,t)\phi_{\alpha}'\left(\frac{x}{\lambda(t)}\right)\,\mathrm{d}x.
\end{equation}

Next, we employ  Lemma \ref{lem2}  to handle the term $A_{4}$. More precisely it yields

\begin{equation}
\begin{split}
|A_{4}(t)|&=\frac{1}{3|\lambda(t)t^{a}\log^{2}t|}\int_{\mathbb{R}}\phi_{\alpha}'\left(\frac{x}{\lambda(t)}\right)|u(x,t)|^3\,\mathrm{d}x\\
&\lesssim_{\|u_{0}\|_{\frac{\alpha+1}{2},2}}\frac{1}{|\lambda(t)t^{a}\log^{2}t|}\int_{\mathbb{R}}\phi_{\alpha}'\left(\frac{x}{\lambda(t)}\right)u^{2}(x,t)\,\mathrm{d}x\\
&=c\left(\|u_{0}\|_{\frac{\alpha+1}{2},2}\right)\frac{1}{t\log t}\int_{\mathbb{R}}\phi_{\alpha}'\left(\frac{x}{\lambda(t)}\right)u^{2}(x,t)\,\mathrm{d}x
\end{split}
\end{equation}
which  in view of \eqref{e10} it  is bounded after integrating in time.

Therefore,
\begin{equation}
|A_{4}(t)|\in L^{1}\left(\left\{t\gg 1 \right\}\right).
\end{equation}

Collecting all the estimates corresponding to this step we conclude that
\begin{equation}\label{e11}
\begin{split}
&\int\limits_{\left\{t\gg 1\right\}}\frac{1}{\lambda(t)t^{a}\log^{2}t}\left\{\int_{\mathbb{R}} \Big(\big(D_{x}^{\frac{\alpha+1}{2}}u\big)^{2}+\big(\hil D_{x}^{\frac{\alpha+1}{2}}u\big)^{2}\Big)(x,t)\,\phi_{\alpha}'\Big(\frac{x}{\lambda(t)}\Big)\mathrm{d}x\right\}\mathrm{d}t\\
&\quad=\int\limits_{\left\{t\gg 1\right\}}\frac{1}{t\log t}\left\{\int_{\mathbb{R}} \Big(\big(D_{x}^{\frac{\alpha+1}{2}}u\big)^{2}+\big(\hil D_{x}^{\frac{\alpha+1}{2}}u\big)^{2}\Big)(x,t)\,\phi_{\alpha}'\Big(\frac{x}{\lambda(t)}\Big)\mathrm{d}x\right\}\mathrm{d}t\\
&\quad <\infty.
\end{split}
\end{equation}
Next, we gather the estimates in \eqref{e10} and \eqref{e11} to conclude that
\begin{equation}\label{e12}
\begin{split}
&\int\limits_{\left\{t\gg 1\right\}}\!\!\frac{1}{t\log t}\left\{\int_{\mathbb{R}} \Big(u^{2}\!+\!\big(D_{x}^{\frac{\alpha+1}{2}}u\big)^{2}\!+\!\big(\hil D_{x}^{\frac{\alpha+1}{2}}u\big)^{2}\Big)(x,t)\,\phi_{\alpha}'\Big(\frac{x}{\lambda(t)}\Big)\mathrm{d}x\right\}\mathrm{d}t\\
&\quad < \infty.
\end{split}
\end{equation}
Since the function $\frac{1}{t\log t}\notin L^{1}\left(\left\{t\gg 1\right\}\right)$ then the condition  \eqref{e10} implies that there exists $(t_{n})_{n}$ an increasing sequence  such that 
\begin{equation}\label{limite}
\int_{\mathbb{R}} \Big(u^{2}+\big(D_{x}^{\frac{\alpha+1}{2}}u\big)^{2}+\big(\hil D_{x}^{\frac{\alpha+1}{2}}u\big)^{2}\Big)(x,t_{n})\,\phi_{\alpha}'\Big(\frac{x}{\lambda(t_{n})}\Big)\mathrm{d}x\longrightarrow 0
\end{equation}
as $n\rightarrow \infty.$

Next, we indicate how to construct this sequence, but first we shall  remind that $\lambda(t)=\frac{t^{b}}{\log t}$, therefore  
\begin{equation}
\begin{split}
\frac{t_{n}^{b(\alpha+2)}}{\big(x^{2}+t_{n}^{2b}\big)^{\frac{\alpha+2}{2}}}\leq\phi_{\alpha}'\Big(\frac{x}{\lambda(t_{n})}\Big)\leq \frac{t_{n}^{b(\alpha+2)}}{\big(x^{2}+t_{n}^{2(b-1)}\big)^{\frac{\alpha+2}{2}}},\quad \mbox{for}\quad t_{n}\gg 1,
\end{split}
\end{equation}
then for $x>0$
\begin{equation}
\begin{split}
\sup_{x\in [2^{k},2^{k+1})}\phi_{\alpha}'\Big(\frac{x}{\lambda(t_{n})}\Big)&\leq \sup_{x\in[2^{k},2^{k+1}) }\frac{t_{n}^{b(\alpha+2)}}{\big(x^{2}+t_{n}^{2(b-1)}\big)^{\frac{\alpha+2}{2}}}\\
&\lesssim c\inf
_{x\in[2^{k},2^{k+1})}\frac{t_{n}^{b(\alpha+2)}}{\big(x^{2}+t_{n}^{2(b-1)}\big)^{\frac{\alpha+2}{2}}}\\
&\lesssim \frac{t_{n}^{b(\alpha+2)}}{\big(x^{2}+t_{n}^{2(b-1)}\big)^{\frac{\alpha+2}{2}}}, \quad \mbox{for}\quad x\in[2^{k},2^{k+1}).
\end{split}
\end{equation}
 Hence, for every $\epsilon>0,$ we consider  $|x|\sim t_{n}^{(b+\epsilon)(\alpha+2)}$, thus for every $k\in \mathbb{N}$ the inequality above can be controlled as follows:
 \begin{equation}
 \begin{split}
 \sup_{x\in[2^{k}, 2^{k+1})}\frac{t_{n}^{b(\alpha+2)}}{\big(x^{2}+t_{n}^{2(b-1)}\big)^{\frac{\alpha+2}{2}}}&\lesssim \frac{1}{t_{n}^{\epsilon(\alpha+2)}}
 \end{split}
 \end{equation}
 so that, if we choose $t_{n}=\log^{\frac{1}{\epsilon(\alpha+2)}}n,$ we get firstly that  $t_{n}$ is an increasing sequence as desired. In the case $x\leq 0$  an analogous argument applies.
   
  Aditionally,   we  introduce a dyadic  partition of unity   
 \begin{equation}
 1=\chi_{0}(x)+\sum_{k=0}^{\infty}\chi\big(2^{-k}x\big)\quad\mbox{where}\quad  \chi_{0}, \,\chi\in C^{\infty}_{0}(\mathbb{R}),
 \end{equation}
 and 
 \begin{equation}
 \mathrm{supp}\,\chi_{0}\subset \{|x|\leq 1\},\qquad  \mathrm{supp\, \chi}\subset\{1/2\leq |x|\leq 2\},
 \end{equation}
 then
 \begin{equation}
 \begin{split}
 &\int_{\mathbb{R}} \Big(u^{2}+\big(D_{x}^{\frac{\alpha+1}{2}}u\big)^{2}+\big(\hil D_{x}^{\frac{\alpha+1}{2}}u\big)^{2}\Big)(x,t_{n})\,\phi_{\alpha}'\Big(\frac{x}{\lambda(t_{n})}\Big)\mathrm{d}x\\
 &=\sum_{k=0}^{\infty} \int_{\mathbb{R}} \chi\big(2^{-k}x\big)\Big(u^{2}+\big(D_{x}^{\frac{\alpha+1}{2}}u\big)^{2}+\big(\hil D_{x}^{\frac{\alpha+1}{2}}u\big)^{2}\Big)(x,t_{n})\,\phi_{\alpha}'\Big(\frac{x}{\lambda(t_{n})}\Big)\mathrm{d}x\\
 &\lesssim \frac{1}{\log n}\sum_{k=0}^{\infty}\int_{\mathbb{R}}\chi\big(2^{-k}x\big)\Big(u^{2}+\big(D_{x}^{\frac{\alpha+1}{2}}u\big)^{2}+\big(\hil D_{x}^{\frac{\alpha+1}{2}}u\big)^{2}\Big)(x,t_{n})\,\mathrm{d}x\\
 &\lesssim\frac{\|u_{0}\|_{\frac{\alpha+1}{2},2}}{\log n}
 \end{split}
 \end{equation}
which proves assertion \eqref{limite}.

In particular, we have shown  that 
\begin{equation}
\lim_{n\uparrow \infty}\int_{|x|\leq \lambda(t_{n})}\Big(u^{2}+\big(D_{x}^{\frac{\alpha+1}{2}}u\big)^{2}+\big(\hil D_{x}^{\frac{\alpha+1}{2}}u\big)^{2}\Big)(x,t_{n})\,\mathrm{d}x= 0.
\end{equation}

\vspace{5mm}

\noindent{\bf Acknowledgements.}  The first and second authors were partially supported by CNPq and FAPERJ/ Brazil.

\medspace

\end{document}